\documentclass[12pt]{amsart}
\usepackage[T1]{fontenc}
\usepackage[utf8]{inputenc}
\usepackage{microtype}
\usepackage{amsmath,amssymb,amsthm,mathrsfs}
\usepackage{dsfont}
\usepackage{esint}
\usepackage{slashed}
\usepackage{mathtools}
\mathtoolsset{showonlyrefs}

\usepackage{hyperref}

\usepackage[margin=1in]{geometry}

\newcommand{\al}{\alpha}
\newcommand{\be}{\beta}
\newcommand{\p}{\partial}
\newcommand{\R}{\mathbb{R}}
\newcommand{\na}{\nabla}
\newcommand{\D}{\slashed{D}}
\newcommand{\pd}{\slashed{\partial}}
\newcommand{\A}{\mathbb{A}}
\newcommand{\Lie}{\mathcal{L}}
\newcommand{\dv}{\dd{vol}}

\newcommand{\MS}[1]{M^{#1}}

\newcommand{\dd}{\mathop{}\!\mathrm{d}}
\DeclareMathOperator{\diverg}{div}

\DeclareMathOperator{\tr}{Trace}

\DeclareMathOperator{\Spin}{Spin}
\DeclareMathOperator{\SO}{SO}
\DeclareMathOperator{\End}{End}
\DeclareMathOperator{\Ker}{Ker}
\DeclareMathOperator{\dist}{dist}

\newtheorem{thm}{Theorem}[section]
\newtheorem{Def}[thm]{Definition}
\newtheorem{lemma}[thm]{Lemma}
\newtheorem{prop}[thm]{Proposition}

\newtheorem{rmk}[thm]{Remark}

\title[A Nonlinear Sigma Model with Gravitino and Partial Regularity]{Partial Regularity for a Nonlinear Sigma Model with Gravitino in Higher Dimensions}

\begin{document}

\author{J\"{u}rgen Jost, Ruijun Wu and Miaomiao Zhu}

\address{Max Planck Institute for Mathematics in the Sciences\\Inselstr. 22--26\\D-04103 Leipzig, Germany}
	\email{jjost@mis.mpg.de}

\address{Max Planck Institute for Mathematics in the Sciences\\Inselstr. 22--26\\D-04103 Leipzig, Germany}
	\email{Ruijun.Wu@mis.mpg.de}

\address{School of Mathematical Sciences, Shanghai Jiao Tong University\\Dongchuan Road 800\\200240 Shanghai, P.R.China}
	\email{mizhu@sjtu.edu.cn}
	
\thanks{The third author was supported in part by National Science Foundation of China (No. 11601325).}

\date{\today}

\begin{abstract}
 We study the regularity problem of the nonlinear sigma model with gravitino fields in higher dimensions. After setting up the geometric model, we derive the Euler--Lagrange equations and consider the regularity of weak solutions defined in suitable Sobolev spaces. We show that any weak solution is actually smooth under some smallness assumption for certain Morrey norms. By assuming some higher integrability of the vector spinor,  we can show a partial regularity result for stationary solutions, provided the gravitino is critical, which means that the corresponding supercurrent vanishes. Moreover, in dimension less than ~6, partial regularity holds for stationary solutions with respect to general gravitino fields.
\end{abstract}



\maketitle

\section{introduction}

Motivated by  super gravity and super string theory in quantum field theory, the nonlinear supersymmetric sigma model has been widely studied in mathematics \cite{deser1976complete, brink1976locally, deligne1999quantum, chen2006dirac, jost2009geometry, jost2014super}. To study the analytical properties of the supersymmetric model, an analogous model was introduced in \cite{jktwz2016regularity} which contains not only the super partner of the scalar field, in mathematical terms a map between Riemannian manifolds,  but also the super partner of the other field of the theory, a Riemannian metric; the latter is a spinorial field called the gravitino. In~\cite{jwz2017coarse, jwz2017energy, jktwz2017symmetries}, further geometrical and analytical properties of the model were explored. These studies also clarified  the role of the gravitino field in the geometry and analysis of this model.
On one side, the gravitino field brings additional  symmetries into the model and hence leads to new conservation laws, making the geometric aspect more interesting; on the other hand, this field is not the solution of  a differential equation since it is only algebraically involved in the action functional, making the analysis of the critical points of the functional subtle.
In physics the gravitino field is known as Rarita-Schwinger field, which appears also in higher dimensional space-time.
In this article,  we shall study the higher dimensional analogue of the model in \cite{jktwz2016regularity}.

In the higher dimensional case, we take a similar action functional to that of \cite{jost2014super,
chen2005regularity}; actually they are of the same form. While this higher dimensional model does not possess a direct physical interpretation, it possesses interesting  analytical properties, which we explore in this paper.

Here we give the precise description of our model. The background material from  spin geometry can be readily found in the literature, for instance in \cite{lawson1989spin, jost2008riemannian}.
Let $(M,g)$ be an $m$-dimensional closed oriented Riemannian manifold, where~$m\ge 3$.
We assume that $(M,g)$  is a spin manifold, with a fixed spin structure given by a 2-fold covering $\xi\colon P_{\Spin}(M,g)\to P_{\SO}(M,g)$ of principal fiber bundles.
Let $S$ be an associated spinor bundle, which is a real vector bundle of rank $2^{2[\frac{m}{2}]}$.
On this spinor bundle $S$ there exist a spinor metric $g_s$ (which is a fiberwise real inner product) and an induced spinor connection $\na^s$ which is compatible with the spinor metric.
As usual we denote the Clifford map by $\gamma\colon TM \to \End(S)$ or sometimes for simplicity by a dot $`` \cdot"$. The Clifford relation for~$(S,M,g)$ reads
\begin{equation}
 \gamma(X)\gamma(Y)+\gamma(Y)\gamma(X)=-2g(X,Y), \qquad \forall X,Y\in\Gamma(TM).
\end{equation}
The spin Dirac operator on $S$ is given as follows. Let $\{e_\al\}$ be an oriented local frame and~$s \in\Gamma(S)$ a spinor field. Then
\begin{equation}
 \pd s\coloneqq \gamma(e_\al)\na^s_{e_\al}s=e_\al\cdot  \na^s_{e_\al}s.
\end{equation}
It is a first-order elliptic operator, which is essentially the Cauchy-Riemann operator in the two-dimensional case, see \cite{lawson1989spin}. It appears also in other models, see for instance \cite{chen2006dirac, jktwz2016regularity}.

A gravitino field is defined as a section of the tensor product bundle $S\otimes TM$. It serves as the supersymmetric partner of the Riemannian metric in physics. Note that the Clifford map induces a surjective map, still denoted by $\gamma$, in the following way:
\begin{equation}
 \begin{split}
  \gamma\colon S\otimes TM &\to S \\
               s\otimes v&\mapsto v\cdot s.
 \end{split}
\end{equation}
The canonical right inverse is given by $\sigma\colon S\to S\otimes TM$ where
\begin{equation}
 \sigma(s)\coloneqq -\frac{1}{m}\delta^{\al\be} e_{\al}\cdot s\otimes e_\be
\end{equation}
with respect to the local $g$-orthonormal frame $\{e_\al\}$. They together give rise to a splitting of the bundle $S\otimes TM$ into $\Ker\gamma\oplus S$,  where the projection maps onto the isomorphic image of ~$S$ and onto $\Ker\gamma$ respectively are given by
\begin{align}
 P=\sigma\circ \gamma, \quad resp. \quad  Q=\mathds{1}-P.
\end{align}
Locally writing $\chi\in\Gamma(S\otimes TM)$ as $\chi=\chi^\al\otimes e_\al$,  we have
\begin{align}
 P\chi=-\frac{1}{m}e_\be\cdot e_\al\cdot \chi^\al\otimes e_\be, \quad Q\chi=-\frac{1}{m}e_\al\cdot e_\be\cdot \chi^\al\otimes e_\be.
\end{align}
We remark that  only the $Q$-part of the gravitino will enter the action functional of our model; for this reason, in the literature sometimes only the sections of the subbundle $Q(S\otimes TM)$ are referred to as gravitinos; however, we will follow the convention to call all sections of~$S\otimes TM$  gravitino fields.

The main nonlinearity of the model comes from the coupling of the spinor with a map between  Riemannian manifolds, which we shall now explain.
Let $(N,h)$ be a Riemannian manifold with the Levi-Civita connection $\na^{N}$.
Consider a smooth map $\phi\colon M\to N$ with the tangent map $T\phi\colon TM\to TN$.
On the pullback there are the induced Riemannian metric~$\phi^*h$ and the induced connection~$\na^{\phi^*TN}$.
Then we can equip the tensor product bundle~$S\otimes \phi^*TN$ with the induced metric $\langle\cdot, \cdot\rangle_{S\otimes\phi^*TN}$ and the connection $\widetilde{\na}\equiv \na^{S\otimes\phi^*TN}$, and define a Dirac operator~$\D$ on~$\Gamma(S\otimes \phi^*TN)$ as follows.
Let~$\{y^i\}$ be local coordinates of~$N$, then~$\{\phi^*(\p_{y^i})\}$ forms a local frame of~$\phi^*TN$. Then~$\psi\in\Gamma(S\otimes\phi^*TN)$ can be locally expressed as~$\psi=\psi^j\otimes \phi^*(\p_{y^i})$. Define
\begin{equation}
 \begin{split}
  \D\psi &\coloneqq e_\al\cdot \widetilde{\na}_{e_\al}\psi  \\
  &= e_\al\cdot \na^s_{e_\al}\psi^j\otimes \phi^*\left(\frac{\p}{\p y^j}\right)
      +e_\al\cdot\psi^j\otimes \na^{\phi^*TN}_{e_\al}\phi^*\left(\frac{\p}{\p y^j}\right) \\
  &= \pd\psi^j\otimes\phi^*\left(\frac{\p}{\p y^j}\right)
     +e_\al\cdot\psi^j\otimes\phi^*\left(\na^{TN}_{T\phi(e_\al)}\frac{\p}{\p y^j}\right).
 \end{split}
\end{equation}
This twisted spin Dirac operator $\D$ is a first-order elliptic differential operator, which is essentially self-adjoint on the Hilbert space $L^2(S\otimes\phi^*TN)$.

The action functional has the same form as in \cite{jktwz2016regularity}:
\begin{equation}
\label{eq:AF}
	\begin{split}
		\A(\phi, \psi;g, \chi)&\coloneqq \int_M |\dd \phi|_{g^\vee\otimes \phi^*h}^2
			+ \langle \psi, \D \psi \rangle_{g_s\otimes \phi^*h} \\
			&\qquad -4\langle (\mathds{1}\otimes\phi_*)(Q\chi), \psi \rangle_{g_s\otimes\phi^*h}
			-|Q\chi|^2_{g_s\otimes g} |\psi|^2_{g_s\otimes \phi^*h}
			-\frac{1}{6} R(\psi) \dd vol_g,
	\end{split}
\end{equation}
where the last curvature term is locally defined by
\begin{equation}
 -\frac{1}{6}R(\psi)
=-\frac{1}{6}R^{N}_{ijkl}\langle \psi^i, \psi^k \rangle_{g_s} \langle \psi^j,\psi^l \rangle_{g_s}.
\end{equation}
As before we write
\begin{equation}
 SR(\psi)\coloneqq \left\langle \psi^l,\psi^j\right\rangle_{g_s}\psi^k\otimes \phi^*\left(R^N\left(\frac{\p}{\p y^k},\frac{\p}{\p y^l}\right)\frac{\p}{\p y^j}\right),
\end{equation}
then $R(\psi)=\langle SR(\psi),\psi\rangle_{g_s\otimes g}$. For later  purposes,  we also introduce the notation
\begin{equation}
 S\na R(\psi)=\phi^*(\na^N R^N)_{ijkl}\langle\psi^i,\psi^k\rangle_{g_s}\langle\psi^j,\psi^l\rangle_{g_s}.
\end{equation}

This action functional is closely related to the functionals for Dirac-harmonic maps and for Dirac-harmonic maps with curvature term. In fact, if the gravitino vanishes in the model, the action $\A$ then reads
\begin{equation}
 L_c(\phi,\psi)=\int_M |\dd\phi|^2+\langle\psi,\D\psi\rangle-\frac{1}{6} R(\psi)\dd vol_g,
\end{equation}
whose critical points are known as Dirac-harmonic maps with curvature term, introduced in \cite{chen2008liouville} and further studied in \cite{branding2015some,jost2015geometric}. And if the curvature term is also dropped, then we get the Dirac-harmonic map functional  introduced in \cite{chen2006dirac}. From the perspective of quantum field theory, they are simplified versions of the model considered here, and describe the behavior of the nonlinear sigma models in  degenerate cases.

\begin{prop}
 The Euler-Lagrange equations for the action functional \eqref{eq:AF} are given by
 \begin{equation} \label{eq:EL}
 \begin{split}
 \tau(\phi)=&\frac{1}{2}R^{N}(\psi, e_\al\cdot\psi)\phi_* e_\al-\frac{1}{12}S\na R(\psi)   \\
 	&  -\frac{2}{m}\left(\langle \na^s_{e_\be}(e_\al \cdot e_\be \cdot \chi^\al), \psi \rangle_S
          + \langle e_\al \cdot e_\be \cdot \chi^\al, \widetilde{\na}_{e_\be} \psi \rangle_S\right),  \\
 \D\psi =& |Q\chi|^2\psi +\frac{1}{3}SR(\psi)+2(\mathds{1}\otimes \phi_*)Q\chi.
 \end{split}
\end{equation}
\end{prop}
The derivation of the equations is quite similar to that in \cite[Section 4]{jktwz2016regularity}, with only minor differences, and thus we omit the proof here.
We can then define weak solutions of the system~\eqref{eq:EL} as the critical points of the action functional~\eqref{eq:AF} on the  Sobolev space
$$W^{1,2}(M,N)\times S^{1,\frac{4}{3}}(\Gamma(S\otimes \phi^*TN)).$$
Recall that, by taking an isometric embedding $N\hookrightarrow \R^K$, the space $W^{1,2}(M,N)$ is defined by
\begin{equation}
 W^{1,2}(M,N)\coloneqq \left\{\phi\in W^{1,2}(M,\R^K) \   \middle|  \  \phi(x)\in N \quad \text{a.e.}  \ x \right\}
\end{equation}
and the space~$ S^{1,\frac{4}{3}}(\Gamma(S\otimes \phi^*TN))$ is defined as the set of weakly differentiable sections~$\psi=(\psi^1,\cdots,\psi^K)\in W^{1,1}(\Gamma(S))^K$ that satisfy
\begin{equation}
 \sum_{i=1}^K \nu_i\psi^i(x)=0, \textnormal{ a.e. } x\in M,\qquad \forall \;\nu=(\nu_1,\cdots, \nu_K)\in (T_{\phi(x)}N)^{\perp}\subset T_{\phi(x)}\R^K,
\end{equation}
and also
\begin{align}
 \psi^i\in L^4(M), \quad\dd\psi^i\in L^{\frac{4}{3}}(M),
\end{align}
for all $1\le i\le K$, see e.g. \cite[Definition 1.1]{wang2009regularity}. Note that we are considering smooth gravitinos throughout this article, thus a pair in this space is sufficient to ensure that the  action is well-defined and finite. Here,  we follow the notation convention on Morrey spaces in \cite{jktwz2016regularity}.

Such a weak solution will be shown to be actually smooth provided a suitable Morrey norm  is small. More precisely, we have the following $\varepsilon$-regularity result.
\begin{thm}\label{thm:full regularity with small Morrey norms}
 For $m\ge 2$, there exists $\varepsilon_0>0$ depending on $(M,g)$, $(N,h)$ and the gravitino~$\chi$ such that if~$(\phi,\psi)\in W^{1,2}(M,N)\times S^{1,\frac{4}{3}}(\Gamma(S\otimes\phi^*TN))$ is a weak solution of \eqref{eq:EL} satisfying
 \begin{equation}\label{eq:small Morrey norms}
  \|d\phi\|_{M^{2,2}(U)}+\|\psi\|_{M^{4,2}(U)}\le \varepsilon_0
 \end{equation}
for some open subset $U\subset M$, then $(\phi,\psi)$ is smooth in $U$.
\end{thm}

In particular, if $U=M$, this says that a weak solutions is smooth provided a certain Morrey norms are small. In dimension two, the condition \eqref{eq:small Morrey norms} can always be satisfied locally by the absolute continuity of the integral, hence weak solutions are always smooth, as proved in~\cite{jktwz2016regularity}.

However, in higher dimensions, the condition \eqref{eq:small Morrey norms} is not always satisfied. From the study of the partial regularity of harmonic maps \cite{Giaquinta82,Giaquinta84,Schoen-Uhlenbeck, evens1991partial, bethuel1993singular}, Dirac-harmonic maps \cite{wang2009regularity} and Dirac-harmonic maps with curvature terms \cite{JLZ} in higher dimensions, we are naturally led to consider the \emph{stationary solutions} of this model. A stationary solution is a weak solution of \eqref{eq:EL} which is also critical with respect to domain variations. Actually this concept supposes the \emph{weak} validity of the conservation law corresponding to diffeomorphism invariance, compare ~\cite{jktwz2017symmetries} and see Definition \ref{def:stationary solutions} for an explicit formulation.

For a stationary solution, we can obtain an ``almost'' monotonicity formula for the map $\phi$ in the sense that
\begin{equation}
 \frac{1}{R_2^{m-2}}\int_{B_{R_2}}|\dd\phi|^2\dv-\frac{1}{R_1^{m-2}}\int_{B_{R_1}} |\dd\phi|^2\dv
 =\int_{R_1}^{R_2} F(r)\dd r
\end{equation}
for some function $F(r)$ which involves all the four fields and is not definite, see Proposition~\ref{monotone}. This prevents us from obtaining a strict monotonicity inequality. However, as pointed out in~\cite{wang2009regularity}, if we assume that $\widetilde{\na}\psi\in L^p(M)$ for some $p>\frac{2}{3}$, then we can  control the error term and get a type of monotonicity inequality of the following form (see Proposition~\ref{prop:almost monotone inequality})
\begin{equation}
 \frac{1}{R_1^{m-2}}\int_{B_{R_1}}|\dd\phi|^2\dd x
  \le \frac{1}{R_2^{m-2}}\int_{B_{R_2}}|\dd\phi|^2\dd x
      +C_0 R_2^{3-\frac{2m}{p}}.
\end{equation}
for some $C_0=C_0(p, \|\widetilde{\na}\psi\|_{L^p},\|\chi\|_{L^\infty})>0$, provided in addition that the gravitino is \emph{critical} with respect to variations, which means that the corresponding supercurrent $J$ (see equation~\eqref{eq:supercurrent}) vanishes, see Section \ref{sect:stationary solutions} and also see \cite{jktwz2017symmetries} for the two dimensional case. Making use of the standard argument, we obtain the partial regularity for stationary solutions, with a singular set of Hausdorff codimension at least two.
\begin{thm}\label{thm:partial regularity for stationary solutions-critical gravitino}
 For $m\ge 3$, let $(\phi,\psi)\in W^{1,2}(M,N)\times S^{1,\frac{4}{3}}(\Gamma(S\otimes\phi^*TN))$ be a stationary solution of the Euler--Lagrange equations \eqref{eq:EL}. If in addition, the gravitino field is critical and ~$\widetilde{\na}\psi\in L^p(M)$ for some $p>\frac{2m}{3}$, then there exists a closed set $\mathcal{S}(\phi)\subset M$ with~$\mathcal{H}^{m-2}(\mathcal{S}(\phi))=0$, such that $(\phi,\psi)\in C^\infty(M\backslash \mathcal{S}(\phi))$.
\end{thm}

We take the gravitino field to be critical for technical reasons. But if the dimension of $M$ is low, we can still get the partial regularity of the stationary solutions for a general gravitino.

\begin{thm}\label{thm:partial regularity for stationary solutions-low dimension}
 For $3\le m\le 5$, let $(\phi,\psi)\in W^{1,2}(M,N)\times \mathcal{S}^{1,\frac{4}{3}}(\Gamma(S\otimes\phi^*TN))$ be a stationary solution of the Euler--Lagrange equations \eqref{eq:EL}.
 If in addition~$\widetilde{\na}\psi\in L^p(M)$ for some $p>\frac{2m}{3}$, then there exists a closed set $\mathcal{S}(\phi)\subset M$ with~$\mathcal{H}^{m-2}(\mathcal{S}(\phi))=0$, such that $(\phi,\psi)\in C^\infty(M\backslash \mathcal{S}(\phi))$.
\end{thm}

The question of partial regularity for stationary solutions with respect to general gravitinos in general dimension remains open.

The article is organized as follows. First we rewrite the Euler--Lagrange equation \eqref{eq:EL} into a local form, where we can explicitly use the structures of the equations. Using a result in \cite{jktwz2016regularity} together with the most recent results developed for critical elliptic systems with an antisymmetric structure developed by Rivi\`{e}re-Struwe~\cite{RS} and further explored in \cite{sharp2014higher, moser2014lpregularity}, we show the full regularity for solutions with some smallness assumptions on certain Morrey norms. Then we consider stationary solutions. We check that this is equivalent to the weak validity of a conservation law. Then we use this formulation to establish the almost monotonicity inequality and finally prove Theorem \ref{thm:partial regularity for stationary solutions-critical gravitino} and Theorem \ref{thm:partial regularity for stationary solutions-low dimension}.

\section{Regularity of weak solutions with small Morrey norms}

In this section we show that the weak solutions of the Euler-Lagrange equations are smooth provided that certain Morrey norms are small, even when the domain is higher dimensional.
The argument differs from that when the domain is two-dimensional, since we no longer  have the  conformal symmetry. In dimension two the smallness of the required Morrey norms came for free, but not in general dimensions.

For simplicity we assume that the equations are located in a Euclidean unit disk $B_1\subset \R^m$ equipped with the standard inner product $g_0$. Then the coordinate frame $\{e_\al=\frac{\p}{\p x^\al}\}$ is orthonormal. The local form of the Euler--Lagrange equations is
\begin{equation}\label{eq:EL-local}
 \begin{split}
  \tau^i(\phi)=& \frac{1}{2}\langle \psi^k, e_\al \cdot \psi^l \rangle e_\al(\phi^j) R^{i}_{\; jkl}
        	-\frac{1}{12}(\na^i R)_{sjkl} \langle \psi^s,\psi^k \rangle \langle \psi^j, \psi^l\rangle  \\
		& -\diverg V^i -\Gamma^{i}_{jk}\langle V^k,\na\phi^j \rangle, \\
   \pd \psi^i  =& -\Gamma^i_{jk}\na \phi^j\cdot\psi^k+|Q\chi|^2\psi^i
		+\frac{1}{3}R^i_{\;jkl}\langle\psi^l,\psi^j\rangle\psi^k
		-e_\al\cdot \na \phi^i \cdot \chi^\al,
 \end{split}
\end{equation}
for $1\le i\le m$, where each $V^i$ is a vector field defined on $M$ via
\begin{equation}\label{eq:def of V}
 \left\langle V^i,W\right\rangle_{g}=\frac{2}{m}\left\langle e_\al\cdot W\cdot \chi^\al, \psi^i\right\rangle_{g_s}, \qquad \forall\; W\in\Gamma(TM).
\end{equation}
\begin{rmk}
Of course, we cannot always transform a given Riemannian metric locally into a Euclidean one. But since regularity is a local issue in the domain, we can consider the metric as a perturbation of the Euclidean one on the unit ball $B_1$, and it is easy to extend the analysis from a Euclidean to such a Riemannian domain. Therefore, for simplicity, we assume that the domain is the unit ball with its Euclidean metric.
\end{rmk}

As  discussed in \cite{jktwz2016regularity} (the calculations there hold true also for general dimensional domains), the equations of the map $\phi$ can be rearranged into a nice form which has an antisymmetric structures, which helps to improve the regularity of weak solutions.
Indeed, isometrically embed $(N,h)$ into some Euclidean space $(\R^K,\delta)$ and denote the push-forwards of $(\phi,\psi)$ by~$(\phi',\psi')$.
Then in \cite{jktwz2016regularity} we have written the equations for the map $\phi$ in the form
\begin{equation}
 \Delta\phi'^i=\Omega^i_j \frac{\p\phi'^j}{\p x^\al}+ Z^i(\psi')-\diverg V'^i,
\end{equation}
for $1\le i\le K$, where the coefficient matrix $\Omega=(\Omega^i_j)$ is antisymmetric: $\Omega^i_j=-\Omega^j_i$, which can be controlled by $|\Omega|\le C(|\na\phi|+|\psi|^2+ |\chi||\psi|)$ and the second term is quartic in $\psi$ in the sense that
\begin{equation}
 |Z^i(\psi')|\le C(N)|\psi'|^4,
\end{equation}
and the vector fields $V'^i$'s are defined in a similar way to \eqref{eq:def of V}.

For the spinors fields $\psi'=(\psi'^1,\cdots,\psi'^K)$,  as  discussed in \cite{jktwz2016regularity}, we have
 \begin{equation}\label{eq-psi'-Componentwisely}
  \begin{split}
   \pd\psi'^a =\;&-\sum_{l,b}\na\phi'^d\cdot\psi'^b\frac{\p\nu_l^b}{\p u^d}\nu_l^a(\phi')+|Q\chi|^2 \psi'^a \\
                 & +\frac{1}{3}\sum_{l,b}\big(\langle\psi'^b,\psi'^d\rangle\psi'^c
                                  -\langle\psi'^c,\psi'^b \rangle\psi'^d\big)
                                \frac{\p \nu_l^b}{\p u^d}(\frac{\p \nu_l}{\p u^c})^{\top,a}
                  -e_\al\cdot \na \phi'^a\cdot\chi^\al.
  \end{split}
 \end{equation}
for each $1\le a\le K$, where $\{\nu_l|l=n+1,\cdots,K\}$ is a local orthonormal frame of~$T^\bot N$ and~$(u^a)_{a=1,\cdots,K}$ are standard global coordinate functions on the Euclidean space $\R^K$. For such Dirac type systems, we know the following fact.
\begin{lemma}[{\cite[Lemma 6.1]{jktwz2016regularity}}]\label{lemma:integrability of spinors}
 Let $B_1\subset \R^m$ be the open unit ball with $m\ge 2$ and let $4<p<\infty$. For a weak solution $\varphi\in \MS{4,2}(B_1, \R^L\otimes \R^K)$ of the nonlinear system
 \begin{equation}
 \pd\varphi^i=A^i_j\varphi^j+B^i, \qquad 1\le i\le K,
 \end{equation}
 where $A\in\MS{2,2}(B_1,\mathfrak{gl}(L,\R)\otimes \mathfrak{gl}(K,\R))$ and $B\in \MS{2,2}(B_1,\R^L\otimes\R^K)$,
 there exists $\varepsilon_0=\varepsilon_0(m,p)>0$ such that if
 \begin{equation}
  \|A\|_{\MS{2,2}(B_1)}\le \varepsilon_0,
 \end{equation}
 then $\varphi\in L^p_{loc}(B_1)$. Moreover, for any $U\Subset B_1$, there holds
 \begin{equation}
  \|\varphi\|_{L^p(U)}\le C(m,p,U)\left(\|\varphi\|_{\MS{4,2}(B_1)}+\|B\|_{\MS{2,2}(B_1)} \right)
 \end{equation}
for some $C(m,p,U)>0$.
\end{lemma}

Thus, assuming that the vector spinors $\psi'\in \MS{4,2}(B_1)$, the above lemma implies that~$\psi'\in L^p_{loc}(B_1)$ for any $p\in [1,\infty)$.

Now we turn to the equations for the map $\phi'$. First recall the following result in \cite{sharp2014higher}, which is an extension of the regularity result in \cite{RS}. Note that the index convention there is different from ours.
\begin{thm}[{\cite[Theorem 1.2]{sharp2014higher}}]\label{thm:Ben}
 For $m\ge 2$, let $u\in W^{1,2}(B_1,\R^K)$ with $\na u\in \MS{2,2}(B_1, \R^m\otimes \R^K)$, $\Omega\in \MS{2,2}(B_1,\mathfrak{so}(K)\otimes \R^m)$ and $f\in L^p(B_1)$ for $\frac{m}{2}<p<m$, weakly solve
 \begin{equation}
  -\Delta u=\Omega\cdot\na u+f.
 \end{equation}
 Then for any $U\Subset B_1$ there exists $\varepsilon_1=\varepsilon_1(m,K,p)>0$ and $C=C(m,K,p,U)>0$ such that whenever $\|\Omega\|_{\MS{2,2}(B_1)}\le \varepsilon_1$ we have
 \begin{equation}
  \|\na^2 u\|_{\MS{\frac{2p}{m},2}}+\|\na u\|_{\MS{\frac{2p}{m-p},2}}\le C(m,K,p,U)\left(\|u\|_{L^1(B_1)}+\|f\|_{L^p(B_1)}\right).
 \end{equation}

\end{thm}
Also note that once $\na u\in \MS{\frac{2p}{m-p},2}(B_1)$, then by the Dirichlet growth theorem we conclude that $u\in C^{2-\frac{m}{p}}_{loc}(B_1)$.

Now we shall deal with our system. We take the initial assumption as follows:
\begin{align}
 \phi'&\in W^{1,2}(B_1, \R^K), \hskip40pt \na \phi'\in \MS{2,2}(B_1,\R^m\otimes\R^K),\\
 \psi'&\in\MS{4,2}(B_1, \R^L\otimes\R^K), \quad \na\psi'\in L^{\frac{4}{3}}(B_1,\R^m\otimes\R^L\otimes\R^K).
\end{align}
Note that $\MS{4,2}(B_1)\subset L^4(B_1)$, so these assumptions are sufficient to ensure the finiteness of the action functional.
Moreover, by Lemma \ref{lemma:integrability of spinors}, we directly get $\psi'\in L^p_{loc}(B_1)$ for any ~$p<\infty$, and it in turn follows that $\psi'\in W^{1,2-o}_{loc}(B_1)$. This leads to $\diverg V'^i\in L^{2-o}_{loc}(B_1)$. Here a function~$f\in L^{p-0}$ means that $f \in L^{q}$ for any $1\leq q<p$, see Definition 1. in \cite{jwz2017coarse}.

In the case $m=\dim {M}=3$, we can apply Theorem \ref{thm:Ben} to conclude that $\na\phi'\in L^{4-o}_{loc}(B_1)$, and hence by a bootstrap argument we can finally obtain that $(\phi',\psi')$ is smooth.

However, if $m=\dim{M}\ge 4$, Theorem \ref{thm:Ben} is not applicable anymore. We then need a result from \cite{moser2014lpregularity}, where the following theorem was shown:

\begin{thm}[{\cite{moser2014lpregularity}}]
 For $m\ge 3$ and $p\in (1,\infty)$ there exists an $\varepsilon_1=\varepsilon_1(m,p)>0$ such that if $u\in W^{1,2}(B_1,N)$ and $f\in L^p(B_1,\R^K)$ satisfy
 \begin{equation}
  \Delta u+A(u)\left(\nabla u,\nabla u\right)=f
 \end{equation}
 in $B_1$ weakly and  $\|\Omega\|_{\MS{2,2}}\le \varepsilon_1$, then $u\in W^{2,p}_{loc}(B_1,\R^K)\cap W^{1,2p}_{loc}(B_1,N)$.
\end{thm}

Here $A$ stands for the second fundamental form of the isometric embedding $N\hookrightarrow \R^K$. 
With these tools in hand, we can get the following conclusion about the regularity of weak solutions of our model.
\begin{thm}
 Suppose $\phi'\in\MS{2,2}(B_1,\R^K)$ and $\psi'\in\MS{4,2}(B_1,\R^L\otimes\R^K)$ satisfy \eqref{eq:EL-local} in a weak sense, and suppose that $\chi$ is smoothly bounded on $B_1$. Then there exists $\varepsilon=\varepsilon(m,\chi)>0$ such that if
 \begin{equation}
  \|\na\phi'\|_{\MS{2,2}(B_1)}+\|\psi'\|_{\MS{4,2}(B_1)}\le \varepsilon,
 \end{equation}
 then $\phi'$ and $\psi'$ are actually smooth on $B_1$.
\end{thm}
By a patching argument, this implies Theorem \ref{thm:full regularity with small Morrey norms}.

\section{Partial regularity for stationary solutions}\label{sect:stationary solutions}

In this section, we shall study stationary weak solutions and show some partial regularity results.

Let us first give the precise definition of stationary solutions to our model. Suppose~$f_t\colon M\to M$ is a family of diffeomorphisms induced by a tangent vector field $X\in\Gamma(TM)$, i.e. $(f_t)_t$ is the induced flow.
As shown in \cite{jktwz2017symmetries} these diffeomorphisms induces transformations on the four arguments of the action functional which leaves the action invariant:
\begin{equation}
 \A(\phi_{f_t},\psi_{f_t};g_{f_t},\chi_{f_t})\equiv \A(\phi,\psi;g,\chi)
\end{equation}
where the concrete expressions of the transformation formulas are given in \cite{jktwz2017symmetries}.
A \emph{stationary} solution is a solution $(\phi,\psi)$ of the Euler--Lagrange equations \eqref{eq:EL} and also
\begin{equation}
 \frac{\dd}{\dd t}\Big|_{t=0} \A(\phi_t,\psi_t;g,\chi)=0
\end{equation}
where
\begin{equation}
 \phi_t\coloneqq \phi\circ f_t\colon (M,g)\to (N,h)
\end{equation}
and
\begin{equation}
 \psi_t\coloneqq \be_t^{-1}\circ F_t^{-1}\circ \psi^j\circ f_t\otimes \phi_t^*\left(\frac{\p}{\p y^j}\right)\in \Gamma(S_g\otimes \phi_t^*TN).
\end{equation}
Here we continue to use the notation introduced in \cite{jktwz2017symmetries}. Note that we have to use the isomorphism $\be_t$ to map the spinors in $S_{f_t^* g}$ to the spinor bundle $S_g$. Now consider the diffeomorphism $f_{-t}$, which is the inverse of $f_t$. By the diffeomorphism invariance we have
\begin{equation}
 \A(\phi_t,\psi_t;g,\chi)=\A(\phi, (\be_{-t}\otimes\mathds{1})\psi;g_{f_{-t}},\chi_{f_{-t}}).
\end{equation}
Note that $\psi_t$ is transformed to
\begin{equation}
 \begin{split}
  F_{-t}^{-1}\circ \psi_t^j\circ f_{-t}& \otimes \phi^*\left(\frac{\p}{\p y^j}\right) \\
 &=\be_{-t}\circ \be_{-t}^{-1}\circ  F_{-t}^{-1}\circ \be_t^{-1}\circ F_t^{-1}\circ \psi^j\circ f_t\circ f_{-t}  \otimes \phi^*\left(\frac{\p}{\p y^j}\right) \\
 &=\be_{-t}\psi^j\otimes \phi^*\left(\frac{\p}{\p y^j}\right)
  \equiv (\be_{-t}\otimes\mathds{1})\psi \in\Gamma(S_{g_{f_{-t}}}\otimes \phi^*TN),
 \end{split}
\end{equation}
while the other fields $g$ and $\chi$ are transformed into
\begin{equation}
 g_{f_{-t}}=f_{-t}^* g
\end{equation}
and
\begin{equation}
 \chi_{f_{-t}}= F_{-t}^{-1}\circ \chi^\al \circ f_{-t}\otimes (Tf_{-t})^{-1} e_\al.
\end{equation}
Therefore, for a stationary solution $(\phi,\psi)$, we have
\begin{equation}
 \begin{split}
  0=& \frac{\dd}{\dd t}\Big|_{t=0} \A(\phi_t,\psi_t;g,\chi) \\
  =& \frac{\dd}{\dd t}\Big|_{t=0} \A(\phi, (\be_{-t}\otimes\mathds{1})\psi;g_{f_{-t}},\chi_{f_{-t}})\\
  =& \int_M -\frac{1}{2}\left\langle \Lie_{-X} g, T\right\rangle
             +\left\langle \Lie^{S_g\otimes TM}_{-X}\chi, J\right\rangle \dv_g
 \end{split}
\end{equation}
where $T$ stands for the energy-momentum tensor of $\A$ and $J$ stands for the supercurrent \cite{jktwz2017symmetries}. Using the adjointness of the divergence operators and the Lie derivative operators, we see that
\begin{equation}
 0=\int_M \langle -X, \diverg_g(T)\rangle +\langle-X, \diverg_\chi(J)\rangle\dv_g.
\end{equation}
Since $X$ can be arbitrary, we conclude that
\begin{equation}\label{eq:conservation law}
 \diverg_g(T)+\diverg_\chi(J)=0.
\end{equation}
\begin{Def}\label{def:stationary solutions}
 A stationary solution is a solution of \eqref{eq:EL} which also satisfies \eqref{eq:conservation law} weakly. More precisely, for any vector field $X\in\Gamma(TM)$,
 \begin{equation}\label{eq:stationary condition}
  \int_M -\frac{1}{2}\langle\Lie_X g, T\rangle+\langle\Lie^{S_g\otimes TM}_X \chi, J\rangle\dv_g=0.
 \end{equation}

\end{Def}

\begin{rmk}
 As in the harmonic map case, strong (say $C^2$) solutions are always stationary, but  general weak solutions need not be stationary.
\end{rmk}

For stationary solutions, we have the following type of monotonicity formula, which helps to obtain the Morrey type control.

\begin{prop}\label{monotone}
 Let $(\phi,\psi)\in W^{1,2}(U, N)\times S^{1,\frac{4}{3}}(U, \R^L\otimes \R^K)$ be a local stationary solution on $U$. Then for any $x_0\in U$ and any $0<R_1\le R_2< \dist(x_0,\p U)$,
 \begin{equation}\label{eq:almost monotone formula}
  \frac{1}{R_2^{m-2}}\int_{B_{R_2}}|\dd\phi|^2\dd x-\frac{1}{R_1^{m-2}}\int_{B_{R_1}}|\dd\phi|^2\dd x
  =\int_{R_1}^{R_2} F(r)\dd r
 \end{equation}
 where $F(r)$ is given by
 \begin{equation}
  \begin{split}
   F(r)\equiv
   &\frac{1}{r^{m-2}}
    \int_{\p B_r}2\left|\frac{\p\phi}{\p r}\right|^2+\langle\psi,\gamma(\p_r)\widetilde{\na}_{\p_r}\psi\rangle
            +\frac{4}{m}\langle\gamma(e_\al)\gamma(\p_r)\chi^\al\otimes\phi_*(\p_r),\psi\rangle \\
   &\qquad\qquad\qquad\qquad\qquad\qquad\qquad
              +\langle\psi,\D\psi\rangle-|Q\chi|^2|\psi|^2-\frac{1}{2}R(\psi)\dd s \\
   -&\frac{1}{r^{m-1}}\int_{B_r} \left((m-1)\langle\psi,\D\psi\rangle
                             +(2-m)|Q\chi|^2|\psi|^2+\frac{4-3m}{6}R(\psi)\right)\dd x \\
   +&\frac{2}{m}\frac{1}{r^{m-2}}\int_{B_r} 2\langle\na^s_{\p_r}\chi^\al\otimes\phi_*e_\be,
                                            \gamma(e_\be)\gamma(e_\al)\psi\rangle
             +|\psi|^2\langle\na^s_{\p_r}\chi^\al,\gamma(e_\be)\gamma(e_\al)\chi^\be\rangle \dd x.
  \end{split}
 \end{equation}
\end{prop}

 \begin{proof}
  For simplicity we assume $U\subset \R^m$ and $x_0=0\in U$.
  For $\epsilon>0$ and $0<r<\dist(0,\p U)$, let $\eta_\epsilon(x)=\eta_\epsilon(|x|)\in C^\infty_0(B_r)$ be a radial cutoff function such that $0\le \eta_\epsilon\le 1$ and $\eta_\epsilon\equiv 1$ on~$B_{r(1-\epsilon)}$.
  Consider the radial vector field
  \begin{equation}
   Y(x)=\eta_\epsilon(|x|)\vec{x}=\eta_\epsilon(|x|)x^\al e_\al.
  \end{equation}
  on $U$, with
  \begin{equation}
   \na_{e_\al}Y^\be= \frac{\p Y^\be}{\p x^\al}
   =\eta_\epsilon(|x|)\delta_{\al\be}+\eta'_{\epsilon}(|x|)\frac{x^\al x^\be}{|x|}.
  \end{equation}
  Note that it is symmetric in $\al$ and $\be$. One can calculate the energy-momentum tensor
  $$T=T_{\al\be}e^\al\otimes e^\be$$
  where
  \begin{equation}\label{eq:super action:energy-momentum}
   \begin{split}
	T_{\al\be}
	&=2\langle \phi_*e_\al,\phi_* e_\be\rangle_{\phi^*h}
	   +\frac{1}{2}\left\langle\psi,\gamma(e_\al)\widetilde{\na}_{e_\be}\psi+\gamma(e_\be)\widetilde{\na}_{e_\al}\psi\right\rangle_{g_s\otimes\phi^*h} \\
	&\quad +\frac{2}{m}\langle\gamma(e_\eta)\gamma(e_\al)\chi^\eta\otimes \phi_*e_\be+\gamma(e_\eta)\gamma(e_\be)\chi^\eta\otimes\phi_*e_\al, \psi\rangle_{g_s\otimes\phi^*h}\\
	&\quad - \left(|\dd\phi|^2_{g^\vee\otimes \phi^*h} + \langle\psi,\D_g\psi\rangle - 4\langle(\mathds{1}\otimes\phi_*)Q\chi, \psi\rangle - |Q\chi|^2 |\psi|^2 -\frac{1}{6}R(\psi)\right)g_{\al\be},
   \end{split}
  \end{equation}
  and the supercurrent
  $$J=J^\al\otimes e_\al$$
  with
  \begin{equation}\label{eq:supercurrent}
   J^\al=\frac{2}{m}\left(2\langle\phi_*e_\be,\gamma(e_\be)\gamma(e_\al)\psi\rangle_{\phi^*h}+|\psi|^2\gamma(e_\be)\gamma(e_\al)\chi^\be\right)
  \end{equation}
  in general dimensions in a manner similar to the two dimensional case considered in \cite{jktwz2017symmetries}.
  Also in this local Euclidean chart, $g_{\al\be}=\delta_{\al\be}$ and
  \begin{equation}
   (\Lie_Y g)_{\al\be}= Y^\mu \frac{\p g_{\al\be}}{\p x^\mu}+ \frac{\p Y^\mu}{\p x^\al} g_{\mu\be}
             +\frac{\p Y^\mu}{\p x^\be} g_{\al\mu}
           =\frac{\p Y^\al}{\p x^\be}+\frac{\p Y^\be}{\p x^\al},
  \end{equation}
 while
  \begin{equation}
   \begin{split}
    \Lie^{S_g\otimes TM}_Y\chi
    =&(\Lie^S_Y \chi^\al)\otimes e_\al+ \chi^\al\otimes (\Lie^{TM}_Y e_\al) \\
    =&\left(\na^s_Y \chi^\al-\frac{1}{4}\gamma(\underbrace{\dd Y^\flat}_{=0})\chi^\al\right)\otimes e_\al
       +\chi^\al\otimes\left(\Lie_Y e_\al+\frac{1}{2}(\Lie_Y g)_\sharp e_\al\right)\\
    =&\na^s_Y \chi^\al\otimes e_\al+\chi^\al\otimes\underbrace{\left(-\na_{e_\al}Y+\frac{\p Y^\be}{\p x^\al}e_\be\right)}_{=0} \\
    =& \na^s_Y \chi^\al\otimes e_\al.
   \end{split}
  \end{equation}
  Next we use the vector field $Y$ in the condition \eqref{eq:stationary condition}. As $T$ is symmetric, and noting that $Y$ is supported in $B_r(0)\subset U$, we have
  \begin{equation}
   \begin{split}
    0=&\int_M -T_{\al\be}\frac{\p Y^\al}{\p x^\be}+\langle\na^s_Y\chi^\al, J^\al\rangle\dv \\
    =& \int_M -\eta_\epsilon(|x|)\tr_g(T)
               -\eta'_{\epsilon}(|x|)|x|\left(\frac{x^\al}{|x|}\frac{x^\be}{|x|}T_{\al\be}\right)
               +\langle\na^s_Y\chi^\al, J^\al\rangle\dv.
   \end{split}
  \end{equation}
 Recall that along the solution of \eqref{eq:EL} it holds that
 \begin{equation}
  \langle\psi,\D\psi\rangle=|Q\chi|^2|\psi|^2+\frac{1}{3}R(\psi)+2\langle(\mathds{1}\otimes\phi_*)Q\chi,\psi\rangle.
 \end{equation}
 Then the trace of the energy-momentum tensor is
 \begin{equation}
  \tr_g(T)=(2-m)|\dd\phi|^2+(m-1)\langle\psi,\D\psi\rangle+(2-m)|Q\chi|^2|\psi|^2+\frac{4-3m}{6}R(\psi).
 \end{equation}
 On the other hand,
 \begin{equation}
  \begin{split}
   \frac{x^\al}{|x|}\frac{x^\be}{|x|}T_{\al\be}
   =&2\left|\frac{\p\phi}{\p r}\right|^2+\langle\psi,\gamma(\p_r)\widetilde{\na}_{\p_r}\psi\rangle
       +\frac{4}{m}\langle\gamma(e_\al)\gamma(\p_r)\chi^\al\otimes\phi_*(\p_r),\psi\rangle \\
     &-\left(|\dd\phi|^2+\langle\psi,D\psi\rangle-4\langle(\mathds{1}\otimes\phi_*)Q\chi,\psi\rangle
             -|Q\chi|^2|\psi|^2-\frac{1}{6}R(\psi)\right)\\
  =&2\left|\frac{\p\phi}{\p r}\right|^2+\langle\psi,\gamma(\p_r)\widetilde{\na}_{\p_r}\psi\rangle
       +\frac{4}{m}\langle\gamma(e_\al)\gamma(\p_r)\chi^\al\otimes\phi_*(\p_r),\psi\rangle \\
     &-\left(|\dd\phi|^2-\langle\psi,\D\psi\rangle+|Q\chi|^2|\psi|^2+\frac{1}{2}R(\psi)\right)\\
  \end{split}
 \end{equation}
 where $\p_r\equiv\frac{\p}{\p r}$ and $r(x)=|x|=|x-x_0|$ denotes the radial function. Since $Y(x)=\eta_\epsilon(|x|)|x|\p_r$, we have
 \begin{multline}
   \langle\na^s_Y \chi^\al,J^\al\rangle
   =\eta_\epsilon(|x|)|x|\langle\na^s_{\p_r}\chi^\al,J^\al\rangle \\
   =\frac{2}{m}\eta_\epsilon(|x|)|x|
     \left(2\langle\na^s_{\p_r}\chi^\al\otimes \phi_*e_\be,\gamma(e_\be)\gamma(e_\al)\psi\rangle
               +|\psi|^2\langle\na^s_{\p_r}\chi^\al,\gamma(e_\be)\gamma(e_\al)\chi^\be\rangle\right).
 \end{multline}
 Substituting these into the stationary condition \eqref{eq:stationary condition}, we get
 \begin{equation}
  \begin{split}
   \int_M &\eta_\epsilon(x)(2-m)|\dd\phi|^2\dv  \\
   &+\int_M \eta_\epsilon(x)\left((m-1)\langle\psi,\D\psi\rangle
            +(2-m)|Q\chi|^2|\psi|^2+\frac{4-3m}{6}R(\psi)\right)\dv  \\
   & +\int_M \eta'_\epsilon(x)|x|\left(2\left|\frac{\p\phi}{\p r}\right|^2
                        +\langle\psi,\gamma(\p_r)\widetilde{\na}_{\p_r}\psi\rangle
                      +\frac{4}{m}\langle\gamma(e_\al)\gamma(\p_r)\chi^\al\otimes\phi_*\p_r,\psi\rangle\right)\dv\\
   &-\int_M \eta'_\epsilon(x)|x|\left(|\dd\phi|^2
                       -\langle\psi,\D\psi\rangle+|Q\chi|^2|\psi|^2+\frac{1}{2}R(\psi)\right)\dv \\
   =&\int_M \frac{2}{m}\eta_\epsilon(x)|x|\left(2\langle\na^s_{\p_r}\chi^\al\otimes\phi_*e_\be,\gamma(e_\be)\gamma(e_\al)\psi\rangle+|\psi|^2\langle\na^s_{\p_r}\chi^\al,\gamma(e_\be)\gamma(e_\al)\chi^\be\rangle\right)\dv.
  \end{split}
 \end{equation}
 Letting $\epsilon\to 0$, this formula reduces to
 \begin{align}
   (2-&m)\int_{B_r} |\dd\phi|^2 \dd x+ r\int_{\p B_r} |\dd\phi|^2 \dd s \allowdisplaybreaks[4]\\
   =& r\int_{\p B_r} 2\left|\frac{\p\phi}{\p r}\right|^2+\langle\psi,\gamma(\p_r)\widetilde{\na}_{\p_r}\psi\rangle
                     +\frac{4}{m}\langle\gamma(e_\al)\gamma(\p_r)\chi^\al\otimes\phi_*(\p_r),\psi\rangle\dd s\allowdisplaybreaks[4]\\
    &+r\int_{\p B_r}\left(\langle\psi,\D\psi\rangle-|Q\chi|^2|\psi|^2-\frac{1}{2}R(\psi)\right)\dd s\allowdisplaybreaks[4] \\
   &-\int_{B_r}\left((m-1)\langle\psi,\D\psi\rangle
                             +(2-m)|Q\chi|^2|\psi|^2+\frac{4-3m}{6}R(\psi)\right)\dd x \allowdisplaybreaks[4]\\
   &+\frac{2}{m}r\int_{B_r} 2\langle\na^s_{\p_r}\chi^\al\otimes\phi_*e_\be,
                                            \gamma(e_\be)\gamma(e_\al)\psi\rangle
             +|\psi|^2\langle\na^s_{\p_r}\chi^\al,\gamma(e_\be)\gamma(e_\al)\chi^\be\rangle \dd x.
 \end{align}
 Note that
 \begin{equation}
  \begin{split}
   \frac{\dd}{\dd r}\left(\frac{1}{r^{m-2}}\int_{B_r}|\dd\phi|^2\dd x\right)
  =&-(m-2)\frac{1}{r^{m-1}} \int_{B_r}|\dd\phi|^2\dd x +\frac{1}{r^{m-2}}\int_{\p B_r}|\dd\phi|^2\dd s \\
  =&\frac{1}{r^{m-1}}\left((2-m)\int_{B_r}|\dd\phi|^2\dd x+ r\int_{\p B_r}|\dd\phi|^2\dd s\right).
  \end{split}
 \end{equation}
 Thus, integrating $r$ from $R_1$ to $R_2$ yields the conclusion.
 \end{proof}

Consequently, we have the following type of monotonicity inequality, which is crucial in the study of partial regularity for harmonic map type problems in higher dimensions.

\begin{prop}\label{prop:almost monotone inequality}
 Assume that $(\phi,\psi)\in W^{1,2}(U,N)\times S^{1,\frac{4}{3}}(U, \R^L\otimes R^K)$ is a stationary solution of \eqref{eq:EL}, and the gravitino is critical, i.e., the supercurrent $J\equiv 0$. If in addition,~$\na\psi\in L^p(U)$ for some $\frac{2m}{3}<p< m$, then there exists $C_0=C_0(p, \|\na\psi\|_{L^p(U)}, \|\chi\|_{L^\infty(U)})>0$ such that for any $x_0\in U$ and $0<R_1<R_2<\min\{\dist(x_0,\p U),1\}$, it holds that
 \begin{equation}\label{eq:almost monotone inequality}
  \frac{1}{R_1^{m-2}}\int_{B_{R_1}}|\dd\phi|^2\dd x
  \le \frac{1}{R_2^{m-2}}\int_{B_{R_2}}|\dd\phi|^2\dd x
      +C_0 R_2^{3-\frac{2m}{p}}.
 \end{equation}

\end{prop}

 \begin{proof}
  From the previous proposition  we have
 \begin{equation}
  \frac{1}{R_1^{m-2}}\int_{B_{R_1}}|\dd\phi|^2\dd x
  =\frac{1}{R_2^{m-2}}\int_{B_{R_2}}|\dd\phi|^2\dd x-\int_{R_1}^{R_2}F(r)\dd r.
 \end{equation}
 Hence it suffices to estimate the last integral, which can be split into three parts
 \begin{equation}
  -\int_{R_1}^{R_2} F(r)\dd r= \mathrm{I}+\mathrm{II}+\mathrm{III},
 \end{equation}
 where
 \begin{align}
   \mathrm{I}\equiv \int_{R_1}^{R_2}&\frac{\dd r}{r^{m-2}}
           \int_{\p B_r} -2\left|\frac{\p\phi}{\p r}\right|^2-\langle\psi,\gamma(\p_r)\widetilde{\na}_{\p_r}\psi\rangle
            -\frac{4}{m}\langle\gamma(e_\al)\gamma(\p_r)\chi^\al\otimes\phi_*(\p_r),\psi\rangle \\
   &\qquad\qquad\qquad\qquad\qquad\qquad\qquad
              -\langle\psi,\D\psi\rangle+|Q\chi|^2|\psi|^2+\frac{1}{2}R(\psi)\dd s ,\allowdisplaybreaks[3]\\
  \mathrm{II}\equiv \int_{R_1}^{R_2}&\frac{\dd r}{r^{m-1}}
             \int_{B_r} \left((m-1)\langle\psi,\D\psi\rangle
                             +(2-m)|Q\chi|^2|\psi|^2+\frac{4-3m}{6}R(\psi)\right)\dd x, \\
  \mathrm{III}\equiv \int_{R_1}^{R_2}&\frac{2}{m}\frac{\dd r}{r^{m-2}}
             \int_{B_r} -\langle\na^s_{\p_r}\chi^\al, J^\al\rangle \dd x.
 \end{align}
 Note that $\psi\in L^4(U)$ and $\na\psi\in L^p(U)$ for some $p\in (\frac{2m}{3},m)$, hence by Sobolev embeddings we have $\psi\in L^{\frac{mp}{m-p}}(U)$, with some $\frac{mp}{m-p}>2m\ge 6$. We also assume that the gravitino field is smoothly bounded. We estimate the other terms as follows:
 \begin{enumerate}
  \item Note that
        \begin{equation}
         -\frac{4}{m}\langle\gamma(e_\al)\gamma(\p_r)\chi^\al\otimes\phi_*(\p_r),\psi\rangle
         \le \left|\frac{\p\phi}{\p r}\right|^2+ C|\chi|^2|\psi|^2
        \end{equation}
        and hence we see that for any $r\in [R_1,R_2]$,
        \begin{equation}
         \begin{split}
          \int_{B_r} -2\left|\frac{\p\phi}{\p r}\right|^2
           &      -\frac{4}{m}\langle\gamma(e_\al)\gamma(\p_r)\chi^\al\otimes\phi_*(\p_r),\psi\rangle
                   +|Q\chi|^2|\psi|^2\dd x \\
           & \le \int_{B_r} -\left|\frac{\p\phi}{\p r}\right|^2 + C|\chi|^2|\psi|^2\dd x
             \le \int_{B_r} C|\chi|^2|\psi|^2\dd x \\
           & \le C\|\chi\|^2_{L^\infty}\|\psi\|^2_{L^{\frac{mp}{m-p}}(B_r)}r^{m(1-\frac{2}{p}+\frac{2}{m})} \\
           & \le C\|\chi\|^2_{L^\infty}\|\psi\|^2_{L^{\frac{mp}{m-p}}(U)}r^{m+2-\frac{2m}{p}}
         \end{split}
        \end{equation}
        which is always bounded by $C\|\chi\|^2_{L^\infty}\|\psi\|^2_{L^{\frac{mp}{m-p}}(U)}R_2^{4-\frac{2m}{p}}r^{m-2}$.
        Using integration by parts we get
        \begin{align}
          &\int_{R_1}^{R_2} \frac{\dd r}{r^{m-2}}
                \int_{\p B_r} -2\left|\frac{\p\phi}{\p r}\right|^2
                  -\frac{4}{m}\langle\gamma(e_\al)\gamma(\p_r)\chi^\al\otimes\phi_*(\p_r),\psi\rangle
                     +|Q\chi|^2|\psi|^2\dd s \\
          =&\int_{R_1}^{R_2}\frac{1}{r^{m-2}}
             \dd\left(\int_{B_r} -2\left|\frac{\p\phi}{\p r}\right|^2
                -\frac{4}{m}\langle\gamma(e_\al)\gamma(\p_r)\chi^\al\otimes\phi_*(\p_r),\psi\rangle
                   +|Q\chi|^2|\psi|^2\dd x\right) \allowdisplaybreaks[4] \\
          =&\frac{1}{R_2^{m-2}}\int_{B_{R_2}}-2\left|\frac{\p\phi}{\p r}\right|^2
                    -\frac{4}{m}\langle\gamma(e_\al)\gamma(\p_r)\chi^\al\otimes\phi_*(\p_r),\psi\rangle
                     +|Q\chi|^2|\psi|^2\dd x \\
           &-\frac{1}{R_1^{m-2}}\int_{B_{R_1}} -2\left|\frac{\p\phi}{\p r}\right|^2
                    -\frac{4}{m}\langle\gamma(e_\al)\gamma(\p_r)\chi^\al\otimes\phi_*(\p_r),\psi\rangle
                     +|Q\chi|^2|\psi|^2\dd x \\
           &+\int_{R_1}^{R_2}\frac{m-2}{r^{m-1}}\int_{B_r} -2\left|\frac{\p\phi}{\p r}\right|^2
                    -\frac{4}{m}\langle\gamma(e_\al)\gamma(\p_r)\chi^\al\otimes\phi_*(\p_r),\psi\rangle
                     +|Q\chi|^2|\psi|^2\dd x  \allowdisplaybreaks[4]\\
          \le& C\|\chi\|^2_{L^\infty}\|\psi\|^2_{L^{\frac{mp}{m-p}}(U)} R_2^{4-\frac{2m}{p}}
               +C\|\chi\|^2_{L^\infty}\|\psi\|^2_{L^{\frac{mp}{m-p}}(U)}
                       \int_{R_1}^{R_2} r^{1-m} r^{m+2-\frac{2m}{p}} \dd r \displaybreak[2]\\
          \le&C\|\chi\|^2_{L^\infty} \|\psi\|^2_{L^{\frac{mp}{m-p}}(U)} R_2^{4-\frac{2m}{p}}.
         \end{align}
  \item For the terms involving derivatives of $\psi$, since we assumed that they have higher integrability,  it holds that
        \begin{equation}
         \begin{split}
          \int_{B_r} -\langle\psi,\gamma(\p_r)\widetilde{\na}_{\p_r}\psi\rangle-\langle\psi,\D\psi\rangle\dd x
          \le C\|\psi\|_{L^{\frac{mp}{m-p}}(B_r)}\|\widetilde{\na}\psi\|_{L^p(B_r)}
              r^{m(1-\frac{2}{p}+\frac{1}{m})}.
         \end{split}
        \end{equation}
        Again using integration by parts we obtain
        \begin{equation}
         \begin{split}
          &\int_{R_1}^{R_2} \frac{\dd r}{r^{m-2}} \int_{\p B_r}
                -\langle\psi,\gamma(\p_r)\widetilde{\na}_{\p}\psi\rangle-\langle\psi,\D\psi\rangle\dd s \allowdisplaybreaks[4]\\
          &=\int_{R_1}^{R_2}\frac{1}{r^{m-2}}
            \dd\left(\int_{B_r}-\langle\psi,\gamma(\p_r)\widetilde{\na}_{\p}\psi\rangle-\langle\psi,\D\psi\rangle \right) \allowdisplaybreaks[4] \\
          &=\frac{1}{R_2^{m-2}} \int_{B_{R_2}}-\langle\psi,\gamma(\p_r)\widetilde{\na}_{\p_r}\psi\rangle
                                                        -\langle\psi,\D\psi\rangle \dd x \allowdisplaybreaks[4]\\
          &\qquad  -\frac{1}{R_1^{m-2}} \int_{B_{R_1}}-\langle\psi,\gamma(\p_r)\widetilde{\na}_{\p_r}\psi\rangle
                                                        -\langle\psi,\D\psi\rangle \dd x \\
          &\qquad +\int_{R_1}^{R_2}\frac{m-2}{r^{m-1}}\int_{B_r}
                                                -\langle\psi,\gamma(\p_r)\widetilde{\na}_{\p_r}\psi\rangle
                                                        -\langle\psi,\D\psi\rangle \dd x \allowdisplaybreaks[3]\\
          &\le C\|\psi\|_{L^{\frac{mp}{m-p}}(U)}\|\widetilde{\na}\psi\|_{L^p(U)}R_2^{3-\frac{2m}{p}}
               +C\|\psi\|_{L^{\frac{mp}{m-p}}(U)}\|\widetilde{\na}\psi\|_{L^p(U)}\int_{R_1}^{R_2} r^{2-\frac{2m}{p}} \dd r \\
         &\le C\left(1+\frac{1}{3-\frac{2m}{p}}\right)\|\psi\|_{L^{\frac{mp}{m-p}}(U)}\|\widetilde{\na}\psi\|_{L^p(U)}R_2^{3-\frac{2m}{p}}.
         \end{split}
        \end{equation}
  \item Next we deal with the curvature term, which is quartic in $\psi$. Note that
        \begin{equation}
         \begin{split}
          \int_{B_r}R(\psi)\dd x\le C(N) \int_{B_r}|\psi|^4\dd x\le C(N) \|\psi\|^4_{L^{\frac{mp}{m-p}}(U)} r^{m(1-\frac{4}{p}+\frac{4}{m})}.
         \end{split}
        \end{equation}
        Hence
        \begin{equation}
         \begin{split}
          &\int_{R_1}^{R_2}\frac{\dd r}{r^{m-2}} \int_{\p B_r} \frac{1}{2}R(\psi)\dd s
          =\int_{R_1}^{R_2}\frac{1}{r^{m-2}}\dd\left(\int_{B_r}\frac{1}{2}R(\psi)\dd x\right) \\
          &=\frac{1}{R_2^{m-2}}\int_{B_{R_2}}\frac{1}{2}R(\psi)\dd x
             -\frac{1}{R_1^{m-2}}\int_{B_{R_1}}\frac{1}{2}R(\psi)\dd x
              +\int_{R_1}^{R_2}\frac{m-2}{r^{m-1}}\int_{B_r}R(\psi)\dd x \\
          &\le C(N)\|\psi\|^4_{L^{\frac{mp}{m-p}}(U)}R_2^{2(3-\frac{2m}{p})}
                +C(N)\int_{R_1}^{R_2} r^{5-\frac{4m}{p}}\dd r \\
          &\le C(N)\left(1+\frac{1}{3-\frac{2m}{p}}\right)
                 \|\psi\|^4_{L^{\frac{mp}{m-p}}(U)}R_2^{2(3-\frac{2m}{p})}.
         \end{split}
        \end{equation}
        For convenience we assumed $R_2<1$ so that the last line above can be bounded by
        \begin{equation}
          C(N)\left(1+\frac{1}{3-\frac{2m}{p}}\right)
                 \|\psi\|^4_{L^{\frac{mp}{m-p}}(U)}R_2^{3-\frac{2m}{p}}.
        \end{equation}

  \item For the summands in $\mathrm{II}$, they can be estimated in a similar, actually easier way as above.
  \item Since  we assumd the gravitino field $\chi$ is critical, i.e. the corresponding supercurrent~$J$ vanishes, we have $\mathrm{III}=0$.

 \end{enumerate}

  Therefore, we have shown that
  \begin{equation}
   -\int_{R_1}^{R_2} F(r)\dd r\le C_0 R_2^{3-\frac{2m}{p}},
  \end{equation}
  where $C_0>0$ depends on $\|\widetilde{\na}\psi\|_{L^p(U)}$, $\|\chi\|_{L^\infty(U)}$ and the value of $p$. Moreover, $C_0\to \infty$ as~$p\searrow\frac{2m}{3}$.
 \end{proof}

\begin{proof}[Proof of Theorem \ref{thm:partial regularity for stationary solutions-critical gravitino}]
 As we have already established Proposition \ref{prop:almost monotone inequality}, the rest of the proof is the same as in the case of Dirac-harmonic maps considered in \cite[Proof of Theorem 1.8]{wang2009regularity}.
\end{proof}

In the proof of Proposition \ref{prop:almost monotone inequality},  we make the assumption that the gravitino field is critical because the term $\mathrm{III}$ would violate the inequality \eqref{eq:almost monotone inequality} in high dimensions. However, in  low dimensions, a slightly different inequality still holds.

\begin{proof}[Proof of Theorem \ref{thm:partial regularity for stationary solutions-low dimension}]
 Consider $\mathrm{III}$ which is given by
\begin{equation}
\int_{R_1}^{R_2}\frac{2}{m}\frac{\dd r}{r^{m-2}} \int_{B_r}
     -2\langle\na^s_{\p_r} \chi^\al\otimes\phi_* e_\be, \gamma(e_\be)\gamma(e_\al)\psi\rangle
     +|\psi|^2\langle\na^s_{\p_r}\chi^\al,\gamma(e_\be)\gamma(e_\al)\chi^\be\rangle \dd x.
\end{equation}
The second summand of the integrand is pretty good, because it can be bounded by
$$C|\na\chi||\chi||\psi|^2,$$
 and it can be estimated  exactly as before. The first summand causes more trouble: since $\na\phi\in L^2$,
\begin{equation}
 \begin{split}
  \int_{B_r} -2\langle\na^s_{\p_r} \chi^\al\otimes\phi_* e_\be, \gamma(e_\be)\gamma(e_\al)\psi\rangle
  &\le C\|\na\chi\|_{L^\infty} \|\na\phi\|_{L^2(U)}\|\psi\|_{L^{\frac{mp}{m-p}}(U)}r^{m(\frac{1}{2}-\frac{1}{p}+\frac{1}{m})}
 \end{split}
\end{equation}
and then
\begin{equation}
 \begin{split}
  \int_{R_1}^{R_2}\frac{2}{m}\frac{\dd r}{r^{m-2}} \int_{B_r}
    & -2\langle\na^s_{\p_r} \chi^\al\otimes\phi_* e_\be, \gamma(e_\be)\gamma(e_\al)\psi\rangle\dd x \\
   \le& C\|\na\chi\|_{L^\infty} \|\na\phi\|_{L^2(U)}\|\psi\|_{L^{\frac{mp}{m-p}}(U)}
         \int_{R_1}^{R_2}r^{3-\frac{m}{2}-\frac{m}{p}}\dd r.
 \end{split}
\end{equation}
We need to ensure that
\begin{equation}
 3-\frac{m}{2}-\frac{m}{p}>-1
\end{equation}
for $p\in (\frac{2m}{3},m)$. If $m\le 5$, this still gives a positive power of $R_2$, which means that we can  still obtain a monotonicity inequality of the form \eqref{eq:almost monotone inequality}, and hence the same conclusion as in Theorem \ref{thm:partial regularity for stationary solutions-critical gravitino} holds. This finishes the proof.
\end{proof}

\end{document}